\documentclass[11pt,twoside,a4paper]{article}
\usepackage{a4wide}
\usepackage{enumitem}
\usepackage{t1enc,lmodern}
\usepackage{latexsym, fancyhdr}
\usepackage{amssymb, amsmath, amsthm}
\usepackage[latin1]{inputenc}
\usepackage{indentfirst}
\usepackage[english]{babel}
\usepackage[bookmarks=true,colorlinks=false,linkcolor=blue]{hyperref}
\usepackage{natbib}

\def\cite{\citet}

\def\cf{{\mathcal F}}

\def\ci{{\mathcal I}}

\def\cl{{\mathcal L}}
\def\cn{{\mathcal N}}

\def\cO{{\mathcal O}}

\def\E{{\mathbb E}}
\def\L{{\mathbb L}}

\def\P{{\mathbb P}}
\def\R{{\mathbb R}}
\def\ind#1{{\mathbf 1}_{\left\{#1\right\}}}

\def\s{\star}

\def\underbar{\underline}
\def\bar{\overline}
\def\abs#1{\mathop{\left| #1 \right|}\nolimits}
\def\norm#1{\mathop{\left\| #1 \right\|}\nolimits}
\def\inv#1{\mathop{\frac{1}{ #1}}\nolimits}
\def\expp#1{\mathop {\mathrm{e}^{ #1}}}
\def\interior#1{\mathop {\mathrm{int}(#1)}}

\emergencystretch=\hsize
\theoremstyle{plain}
\newtheorem{thm}{Theorem}[section]
\newtheorem{prop}[thm]{Proposition}
\newtheorem{lemma}[thm]{Lemma}

\makeatletter
\newcounter{hypo}
\newcommand*{\dohypo}{\textbf{(A\thehypo)}}
\newenvironment{hypo}{%
  
  \refstepcounter{hypo}
  \list{}{%
    \settowidth{\labelwidth}{\dohypo}%
    \setlength{\labelsep}{10pt}%
    \setlength{\leftmargin}{\labelwidth}
    \advance\leftmargin\labelsep%
  }%
\item[\dohypo]%
}{%
  \endlist
}
\def\hypref#1{\hyperref[hyp:#1]{(\textbf{A\ref*{hyp:#1}})}}
\def\hypreff#1#2{\hyperref[hyp:#2]{(\textbf{A\ref*{hyp:#1}-\ref*{hyp:#2}})}}

\makeatother

\theoremstyle{definition}
\newtheorem{rem}{Remark}

\author{J\'er\^ome Lelong\footnote{Laboratoire Jean Kuntzmann, Universit\'e de
    Grenoble et CNRS, BP 53, 38041 Grenoble C\'edex 9, FRANCE, 
    e-mail : jerome.lelong@imag.fr}}
\date{\today}

\hypersetup{linkcolor=black,colorlinks=true,citecolor=black}

\begin{document}

\title{Asymptotic normality of randomly truncated stochastic algorithms}

\maketitle

\begin{abstract} We study the convergence rate
  of randomly truncated stochastic algorithms, which consist in the
  truncation of the standard Robbins-Monro procedure on an increasing sequence
  of compact sets. Such a truncation is often required in practice to ensure
  convergence when standard algorithms fail because the expected-value function
  grows too fast. In this work, we give a self contained proof of a
  central limit theorem for this algorithm under local assumptions on the
  expected-value function, which are
  fairly easy to check in practice.
\end{abstract}

\section{Introduction}

The use of stochastic algorithms is widespread for solving  stochastic
optimization problems. These algorithms are extremely valuable for a practical
use and particularly well suited to localize the zero of a function $u$.
Such algorithms go back to the pioneering work of \cite{MR0042668}, who
considered the sequence
\begin{equation}
  \label{eq:rm}
  X_{n+1} = X_n -\gamma_{n+1} u(X_n) - \gamma_{n+1} \delta M_{n+1}
\end{equation}
to estimate the zero of the function $u$. The sequence $(\gamma_n)_n$
classically denotes the gain or step sequence of the algorithm and $(\delta
M_n)_n$ depicts a random measurement error. Nevertheless, the assumptions
required to ensure the convergence \textemdash\ basically, a sub-linear growth
of $u$ on average \textemdash\ are barely satisfied in practice, which
dramatically reduces the range of applications.
\cite{chen86:_stoch_approx_proced} proposed a modified algorithm to deal with
fast growing functions. Their new algorithm can be summed up as
\begin{equation}
  \label{eq:chen-intro}
  X_{n+1} = X_n -\gamma_{n+1} u(X_n) - \gamma_{n+1} \delta M_{n+1} +
  \gamma_{n+1} p_{n+1}
\end{equation}
where $(p_n)_n$ is a truncation term ensuring that the sequence $(X_n)_n$ cannot
jump too far ahead in one step.

In this paper, we are concerned with the rate of convergence of
Equation~\eqref{eq:chen-intro}. Numerous results are known for the sequence
defined by Equation~\eqref{eq:rm}, which is known to converge at the rate
$\sqrt{\gamma_n}$ when $\gamma_n$ is of the form $\frac{\gamma}{n^\alpha}$
with $1/2<\alpha \le 1$ (see \cite{delyon96:_gener_resul_conver_stoch_algor},
\cite{duflo97:_random_iterat_model} or \cite{MR1882723} for instance).  When
$\gamma_n = \frac{\gamma}{n}$ and the Hessian matrix at the optimum is of the
form $\lambda I$, \cite{duflo97:_random_iterat_model} showed that the
convergence rate depends on the relative position of $\lambda$ and
$\frac{\gamma}{2}$.  A functional central limit theorem for this algorithm was
proved by \cite{bouton85:_approx_gauss_markov} and \cite{benveniste87:_algor}.
The convergence rate of constrained algorithms was studied by
\cite{harold03:_stoch_appor_recur_algor_applic}. The problem of multiple
targets was tackled by \cite{pelletier98:_weak} who proved a Central Limit
Theorem.  However, very few results are known about the convergence rate of
the algorithm devised by \cite{chen86:_stoch_approx_proced}. \cite{MR1942427}
briefly studied the convergence rate under global hypotheses on the noise
sequence $(\delta M_n)_n$.  Here, we aim at giving a clarified, self-contained
and elementary proof of this result under local assumptions (see
Section~\ref{sec:compare-chen} for a detailed comparison of the two results).
Besides giving a clarified and self-contained proof of the central limit
theorem for randomly truncated algorithm, the improvement brought by our work
is the use of the local condition $\sup_n \E[ \abs{\delta
  M_n}^{2+\rho}\ind{|X_{n-1} - x^\s| \le \eta}] < \infty$ with some $\rho>0$
and $\eta>0$ replacing the global condition $\sup_n \E[ \abs{\delta
  M_n}^{2+\rho}] < \infty$. 

First, we define the general framework and explain the algorithm developed by
\cite{chen86:_stoch_approx_proced}. Our main results are stated in
Theorems~\ref{thm:tcl-chen} and~\ref{thm:tcl2-chen} (see
page~\pageref{thm:tcl-chen}) depending on the decreasing speed of the
sequence $(\gamma_n)_n$. In Section~\ref{sec:compare-chen}, we discuss the
improvements brought by our new results and we give a concrete example to show
the benefits of using local assumptions. Section~\ref{sec:proofs-tcl} is
devoted to the proof of the main results.

\section{A CLT for randomly truncated stochastic algorithms}

It is quite common to look for the root of a continuous function $u \colon x
\in \R^d \longmapsto u(x) \in \R^d$, which is not easily tractable. We assume
that we can only access $u$ up to a measurement error embodied in the
following by the sequence $(\delta M_n)_n$ and that the norm $\abs{u(x)}^2$
grows faster than $\abs{x}^2$ such that the standard Robbins-Monro algorithm
(see Equation~\eqref{eq:rm}) quickly fails. Instead, we consider the
alternative procedure introduced by \cite{chen86:_stoch_approx_proced}. This
technique consists in forcing the algorithm to remain in an increasing
sequence of compact sets $(K_j)_j$ such that
\begin{equation*}
  \bigcup_{j=0}^{\infty} K_j \: = \: \R^d \quad \mbox{and} \quad \forall j,
  \; K_j \varsubsetneq \interior{K_{j+1}}.
\end{equation*}
It prevents the algorithm from blowing up during the first iterates.  Let
$(\gamma_n)_n$ be a decreasing sequence of positive real numbers satisfying
$\sum_n \gamma_n = \infty$ and $\sum_n \gamma_n^2 < \infty$.  For $X_0
\in \R^d$ and $\sigma_0 = 0$, we define the sequences of random variables
$(X_n)_n$ and $(\sigma_n)_n$ by
\begin{equation}
  \label{eq:chen}
  \begin{cases}
    & X_{n + \frac{1}{2}}  = X_{n} - \gamma_{n+1} u(X_n) - \gamma_{n+1} \delta M_{n+1},\\
    \text{if $X_{n + \frac{1}{2}} \in K_{\sigma_n}$} & X_{n+1} =
    X_{n + \frac{1}{2}} \quad \mbox{ and } \quad \sigma_{n+1} = \sigma_n, \\
    \text{if $X_{n + \frac{1}{2}} \notin K_{\sigma_n}$} & X_{n+1}
    = X_{0}  \quad \mbox{ and } \quad \sigma_{n+1} = \sigma_n + 1. 
  \end{cases}
\end{equation}
Let $\cf_n$ denote the $\sigma-$algebra generated by $(\delta M_k, k \le n)$,
$\cf_n=\sigma(\delta M_k, k \le n)$. We assume that $(\delta M_n)_n$ is a
sequence of martingale increments, i.e. $\E(\delta M_{n+1} | \cf_n) = 0$.

\begin{rem} $X_{n+\inv{2}}$ is actually drawn from the dynamics of the
  Robbins-Monro algorithm (see Equation~\eqref{eq:rm}). If the standard
  algorithm wants to jump too far ahead it is reset to a fixed value. When
  $X_{n + \frac{1}{2}} \notin K_{\sigma_n}$, one can set $X_{n+1}$ to any
  measurable function of $(X_0, \dots, X_n)$ with values in a given compact
  set. The existence of such a compact set is crucial to prove the
  a.s. convergence of ${(X_n)}_n$.
\end{rem}

\noindent It is more convenient to rewrite Equation~\eqref{eq:chen} as
follows
\begin{equation}
  \label{eq:chen2}
  X_{n+1} = X_{n}  - \gamma_{n+1} u(X_n) -
  \gamma_{n+1}\delta M_{n+1} + \gamma_{n+1} p_{n+1}  
\end{equation}
where
\begin{equation*}
  p_{n+1} =  \left(u(X_n) + \delta M_{n+1} + \frac{1}{\gamma_{n+1}} (X_0 - X_n) \right)
  \ind{X_{n+\frac{1}{2}} \notin K_{\sigma_n}}.
\end{equation*}
In this paper, we only consider gain sequences of the type $\gamma_n =
\frac{\gamma}{(n+1)^\alpha}$, with $1/2 < \alpha \leq 1$.  If $\alpha = 1$, we
obtain a slightly different limit. For values of $\alpha$ outside this range,
the almost sure convergence is not even guarantied.

\subsection{Hypotheses}

In the following, the prime notation stands for the transpose operator. We
introduce the following hypotheses. 
\begin{hypo}
  \label{hyp:regularity}
  \begin{enumerate}[label={\it \roman{*}.}, ref=\roman{*}]
  \item \label{hyp:convexe}  
    $\exists x^\s \in \R^d$ s.t. $u(x^\s) = 0$ and
    $\forall x \in \R^d,\: x \neq x^{\star}, \; (x-x^{\star}) \cdot u(x) >0$. 

  \item  \label{hyp:C-1} There exist a function $y: \R^d \rightarrow \R^{d \times d}$
    satisfying $\lim_{\abs{x} \rightarrow 0} \abs{y(x)} = 0$ and a symmetric
    positive definite matrix $A$ such that
    \begin{equation*}
    u(x) = A (x - x^{\star}) + y(x - x^{\star}) (x - x^{\star}).
  \end{equation*}
\end{enumerate}
\end{hypo}

\begin{hypo}
  \label{hyp:cv-ps}
  For any $q >0$, the series $\sum_{n} \gamma_{n+1} \delta M_{n+1} \ind{|X_n -
    x^{\s}| \le q}$ converges almost surely.
\end{hypo}

\begin{hypo}
  \label{hyp:chen-mart-cv}
  \begin{enumerate}[label={\it \roman{*}.}, ref=\roman{*}]
  \item \label{hyp:ui}
    There exist two real numbers $\rho>0$ and $\eta>0$ such that
    \begin{equation*}
      \kappa = \sup_n \E\left( \abs{\delta M_n}^{2+\rho}
        \ind{\abs{X_{n-1}-x^{\star}} \leq \eta}\right) < \infty.
    \end{equation*}

  \item \label{hyp:cv_bracket}
    There exists a symmetric positive definite matrix $\Sigma$ such that
    \begin{equation*} \E\left(\delta M_n \delta M_n^\prime \big| \cf_{n-1} \right)
      \ind{\abs{X_{n-1}-x^{\star}} \leq \eta}\xrightarrow[n \rightarrow
      \infty]{\P} \Sigma.
    \end{equation*}
    
  \end{enumerate}
\end{hypo}

\begin{hypo}
  \label{hyp:chen-interior} There exists $\mu>0$ such that $\forall n \geq 0,  \;
  d(x^{\star}, \partial K_n) \ge  \mu $.
\end{hypo}

\begin{rem}\label{rem_hypo}
  {\it Comments on the assumptions.}
  \begin{enumerate}
  \item Hypothesis \hypreff{regularity}{convexe} is satisfied as soon as $u$ can
    be interpreted as the gradient of a strictly convex function. The Hypothesis
    \hypreff{regularity}{C-1} is equivalent to saying that $u$ is differentiable
    at~$x^{\star}$.
  \item Hypothesis \hypref{cv-ps} ensures that $X_n \longrightarrow x^\s$
    a.s. and $\sigma_n$ is almost surely finite, see
    \cite{lelongar:_almos_sure_conver_of_random} for a proof of this result.
  \item Hypothesis \hypreff{chen-mart-cv}{ui} corresponds to some local uniform
    integrability condition and  reminds of Lindeberg's
    condition. \hypreff{chen-mart-cv}{cv_bracket} guaranties the convergence of
    the angle bracket of the martingale of interest. 
  \item Hypothesis \hypref{chen-interior} is only required for technical
    reasons but one does not need to be concerned with it in practical
    applications. It reminds of the case of constrained stochastic algorithms
    for which the CLT can only be proved for non saturated constraints.
  \end{enumerate}
\end{rem}

\subsection{Main results}
\label{sec:chen-tcls}

For $n \geq 0$, we define the renormalized and centered error
\begin{equation*}
  \Delta_n = \frac{X_n - x^{\star}}{\sqrt{\gamma_n}}.
\end{equation*}

\paragraph{A CLT for $1/2 < \alpha < 1$}

\begin{thm}
  \label{thm:tcl-chen} If we assume Hypotheses \hypref{regularity} to
  \hypref{chen-interior}, the sequence ${(\Delta_{n})}_{n}$ converges in
  distribution to a normal random variable with mean $0$ and covariance
  \begin{equation*}
    V =  \int_0^{\infty} \exp{(-A t)} \Sigma \exp{(- A t)} dt. 
  \end{equation*}
\end{thm}

\paragraph{A CLT for  $\alpha = 1$}

\begin{thm}
  \label{thm:tcl2-chen} We assume Hypotheses \hypref{regularity} to
  \hypref{chen-interior} and
  \begin{hypo}
    \label{hyp:A-gamma-definite}
    $\gamma A - \frac{1}{2} I$ is positive definite.
  \end{hypo}
  Then, the sequence ${(\Delta_{n})}_{n}$ converges in distribution to a
  normal random variable with mean $0$ and covariance
  \begin{equation*}
    V =  \gamma \int_0^{\infty} \exp{\left(\left(\frac{I}{2}-\gamma
          A\right)t\right)} \Sigma \exp{\left(\left(\frac{I}{2}- \gamma
          A\right)t\right)} dt. 
  \end{equation*}
\end{thm}

\begin{rem} Hypothesis \hypref{A-gamma-definite} involves the gradient of
  function $u$ at the point $x^{\star}$, which is seldom tractable from a
  practical point of view but one can definitely not avoid it. The positivity of
  $\gamma A - \frac{1}{2} I$ is the border of two different convergence regimes
  as already noted by \cite{duflo97:_random_iterat_model} for the Robbins-Monro
  algorithm.
\end{rem}

\subsection{Discussion around the assumptions of Theorem~\ref{thm:tcl2-chen}}
\label{sec:compare-chen}

Theorem~\ref{thm:tcl2-chen} is actually an extension of \cite[Theorem
3.3.1]{MR1942427}. The main improvements brought by our new result concern the
conditions imposed on the noise term. Our Assumption~\hypref{chen-mart-cv}  is
weaker than the one imposed by \cite[A3.3.3]{MR1942427} since we only assume
local conditions on the noise terms; namely, unlike Chen, we only need to
monitor the behavior of $(\delta M_n)_n$ in a small neighborhood of the optimum
$x^\s$ (see Assumptions~\hypreff{chen-mart-cv}{ui} and
\hypreff{chen-mart-cv}{cv_bracket}). Moreover, we only assume the local
convergence in probability of the angle bracket of the martingale of interest
built with $(\delta M_n)_n$ whereas \cite[Equation (3.3.22)]{MR1942427} requires
the almost sure convergence which may be a little harder to prove in practical
applications.

Concerning Assumption~\hypreff{regularity}{C-1}, it essentially means that $u$
must be differentiable at $x^\s$. This is to be compared to the Hölder
continuity property of the  remainder of the first order expansion of $u$ at
$x^\s$ required by \cite[A3.3.4]{MR1942427}, which is not so obvious to check in
practice. Our goal in this work was not only to state a theorem with weaker
assumptions but also to present a self contained and elementary proof of a
central limit theorem for truncated stochastic algorithms. In particular,
Lemma~\ref{lem:tight} provides a smart way of handling the truncation terms.

Let us us consider an example which often arises in practice (see for instance
\cite{arouna03:_robbin_monro} and \cite{lelong07:_etude_asymp_des_algor_stoch}).
Assume the function $u$ is defined as an expectation $u(x) = \E(U(x, Z))$ where
$Z$ is a random vector, then we can for instance take $\delta M_{n+1} = U(X_n,
Z_{n+1}) - u(X_n)$ with $(Z_n)_n$ an i.i.d. sequence of random vectors following
the law of $Z$. With this choice and if we further assume that for all $q >0$,
$\sup_{|x| \le q} \E(|U(x, Z)|^{2+\rho}) < \infty$ and that the function $x
\longmapsto \E(U(x,Z) U(x,Z)')$ is continuous at $x^\s$, then it is obvious that
Assumptions~\hypref{cv-ps} and~\hypref{chen-mart-cv} are satisfied. Note that in
this particular but widely used setting, there is no assumption to be checked
along the paths of the algorithm as it was the case in the results of Chen. This
considerably widens the range of applications as the assumptions of our theorems
boils down to basic regularity properties of the function $U$. Note also that we
do not impose any condition on the behaviour of the sequence $(p_n)_n$, i.e. on
the choice of the compact subsets $(K_n)_n$.

\section{Proofs of  Theorems \ref{thm:tcl-chen} and \ref{thm:tcl2-chen}}
\label{sec:proofs-tcl}

In this section, we prove the Theorems presented in Section~\ref{sec:chen-tcls}
through a series of three lemmas. The proofs of these lemmas are postponed to
Section~\ref{sec:lemmas-proof}.

\subsection{Technical lemmas}

For any fixed $n>0$, we introduce  $s_{n,k} = \sum_{i=0}^k \gamma_{n+i}$ for
$k \ge 0$ and we set $s_{n,0} = 0$. $(s_{n,k})_{k \ge 0}$ can be
interpreted as a discretisation grid of $[0,\infty)$ because $\lim_{k
  \rightarrow \infty} s_{n,k} = \infty$.

Theorems~\ref{thm:tcl-chen} and \ref{thm:tcl2-chen} are based on the
following three lemmas.
\begin{lemma}
  \label{lem:tight} Let $\varepsilon>0$ and $\eta >0$ as in
  Hypothesis~\hypref{chen-mart-cv}. There exists $N_0>0 $, such that if we
  define for $n \ge N_0$
  \begin{equation*}
    A_n = \left\{\sup_{n \geq m \geq N_0} \abs{X_m-x^{\star}} \le \eta \right\},
  \end{equation*}
  then
  \begin{equation*}
    \P(A_n)\ge 1-\varepsilon \quad \forall n \ge  N_0 \quad \text{and}\quad \sup_{n \geq
      N_0} \E\left( \abs{\Delta_n}^2 \ind{A_n}\right) < \infty.
  \end{equation*}
\end{lemma}

\begin{lemma}
  \label{lem:delta-n-expr} For any integers $t>0$ and $n>0$
  \begin{eqnarray}
    \label{eq:deltanp}
    \Delta_{n+t} &  = &  e^{-s_{n,t} Q} \Delta_n
    -  \sum_{k=0}^{t-1} \expp{Q(s_{n,k}-s_{n,t})} \sqrt{\gamma_{n+k+1}} \delta
    M_{n+k+1} -  \sum_{k=0}^{t-1} \expp{Q(s_{n,k}-s_{n,t})} \gamma_{n+k} R_{n+k},
  \end{eqnarray}
  where 
  \begin{itemize}
    \item if $\alpha = 1$,
      \begin{equation}
        \begin{cases}
          Q & = A - \frac{1}{2\gamma} I \\  
          R_{m} & = - y(X_{m}-x^{\star}) \Delta_{m}  
           + \frac{1}{\sqrt{\gamma_{m+1}}}  p_{m+1} \\
           & \qquad +  \gamma_{m}
          (a_{m} I + b_{m} (A +y(X_{m}-x^{\star})) + \cO(\gamma_m)) \Delta_{m}, 
        \end{cases}
        \label{eq:reste1}
      \end{equation}
    \item if $1/2 < \alpha < 1$,
      \begin{equation}
        \begin{cases}
          Q & = A  \\  
          R_m & = y(X_{m}-x^{\star})) \Delta_m -
          \frac{1}{\sqrt{\gamma_{m+1}}} p_{m+1} \\
          & \qquad - \frac{1}{m \gamma_m} (a_m I + b_m 
          \gamma_n (A + y(X_{m}-x^{\star}))) \Delta_m + \cO (\gamma_m) \Delta_{m}
        \end{cases}
        \label{eq:restealpha}
      \end{equation}
      with $(a_n)_n$ and $(b_n)_n$ two real valued and bounded sequences.
  \end{itemize}
  Moreover, the last term in~\eqref{eq:deltanp} tends to zero in probability.
\end{lemma}

\begin{lemma}
  \label{lem:int-sto-conv} In Equation~(\ref{eq:deltanp}), the sequence
  $(\sum_{k=0}^t \expp{Q(s_{n,k}-s_{n,t})} \sqrt{\gamma_{n+k}} \delta
  M_{n+k})_t$ converges in distribution  to $\cn(0,V_n)$ for any fixed $n$ when
  $p$ goes to infinity, where $V_n = \sum_{k=0}^\infty
  \gamma_{n+k}\expp{-Qs_{n,k}} \Sigma \expp{-Qs_{n,k}}$.
\end{lemma}

\begin{proof}[Proof of Theorems \ref{thm:tcl-chen} and \ref{thm:tcl2-chen}]
  Let us consider Equation~(\ref{eq:deltanp}) for a fixed $n > N_0$, where $N_0$
  is defined in Lemma~\ref{lem:tight}. Because the matrix $Q$ is definite
  positive and $\Delta_n$ is almost surely finite, $e^{-s_{n,t} Q} \Delta_n$
  tends to zero almost surely when $t$ goes to infinity. Thanks to
  Lemma~\ref{lem:delta-n-expr}, the last term in Equation~\eqref{eq:deltanp}
  tends to zero in probability when $t$ goes to infinity.

  Combining these two convergences in probability to zero with
  Lemma~\ref{lem:int-sto-conv} yields the convergence in distribution of
  $(\Delta_{n+t})_t$ to a normal random variable with mean $0$ and variance $V$
  when $p$ goes to infinity, where $V$ is defined in
  Lemma~\ref{lem:int-sto-conv}. Plugging the value of the matrix $Q$ (see
  Equations~(\ref{eq:reste1}) and~(\ref{eq:restealpha})) in the expression of
  $V$ yields the result.
\end{proof}
Note that the proof for the classical Robbins Monro algorithm is much
simpler since we do not need to introduce the $A_n$ sets, which are only used
here to handle the truncation terms.

\subsection{Proofs of the lemmas}
\label{sec:lemmas-proof}

\subsubsection{Proof of Lemma \ref{lem:tight}}

We only do the proof in the case $\alpha=1$, as in the other case, it is
sufficient to slightly modify a few Taylor expansions and the same results still
hold.  From Equation~(\ref{eq:chen2}), we have the following recursive relation
\begin{equation*}
  \Delta_{n+1} = \frac{X_{n+1} - x^{\star}}{\sqrt{\gamma_{n+1}}}
  = \sqrt{\frac{\gamma_n}{\gamma_{n+1}}} \: \Delta_n -
  \sqrt{\gamma_{n+1}} (u(X_n) + \delta M_{n+1} - p_{n+1}).
\end{equation*}
Using Hypothesis \hypreff{regularity}{C-1}, the previous equation becomes
\begin{equation}
  \label{eq:2}
  \Delta_{n+1} = \displaystyle \left(\sqrt{\frac{\gamma_n}{\gamma_{n+1}}} I -
    \sqrt{\gamma_{n+1} \gamma_n} (A + y(X_n-x^{\star}))\right) \Delta_n
  -  \sqrt{\gamma_{n+1}}\delta M_{n+1} + \sqrt{\gamma_{n+1}} p_{n+1}.  
\end{equation}
The following Taylor expansions hold
\begin{equation}
  \label{eq:8}
  \sqrt{\frac{\gamma_n}{\gamma_{n+1}}} =  1 + \frac{\gamma_n}{2 \gamma} 
  +\cO\left(\gamma_n^2\right) \mbox{ and } 
  \sqrt{\gamma_n \gamma_{n+1}} = \gamma_n + \cO\left(\gamma_n^2\right).
\end{equation}
There exist two real valued and bounded sequences $(a_n)_n$ and $(b_n)_n$ such that
\begin{equation*}
  \sqrt{\frac{\gamma_n}{\gamma_{n+1}}} =  1 + \frac{\gamma_n}{2 \gamma} 
  +\gamma_n^2 a_n \mbox{ and } 
  \sqrt{\gamma_n \gamma_{n+1}} = \gamma_n + \gamma_n^2 b_n.
\end{equation*}
This enables us to simplify Equation (\ref{eq:2})
\begin{align}
  \label{eq:3}
   \Delta_{n+1} =  &  \Delta_n - \gamma_n  Q \Delta_n   - \gamma_n
   y(X_n-x^{\star}) \Delta_n  - \sqrt{\gamma_{n+1}}\delta M_{n+1} \nonumber\\ 
  & + \sqrt{\gamma_{n+1}}  p_{n+1} +  \gamma_{n}^2
  (a_n I + b_n (A +y(X_n-x^{\star}))) \Delta_n, 
\end{align}
where $Q = A-\frac{I}{2 \gamma}$.
Let $\displaystyle  \Delta_{n+\frac{1}{2}} = \frac{X_{n+ \frac{1}{2}} -
  x^{\star}}{\sqrt{\gamma_{n+1}}}$, where  $X_{n+ \frac{1}{2}}$, defined by
Equation~\eqref{eq:chen}, is the value of the new iterate obtained before
truncation.
\begin{align*}
  \abs{\Delta_{n+\frac{1}{2}}}^2 = & 
  \left| \Delta_n - \gamma_n Q \Delta_n - \gamma_n y(X_n-x^{\star}) \Delta_n -
    \sqrt{\gamma_{n+1}}\delta M_{n+1} \right.\\
   & \quad \left.   + \gamma_{n}^2 (a_n I + b_n (A +y(X_n-x^{\star}) I))
     \Delta_n\right|^2 \\
   \le & | \Delta_n - \gamma_n Q \Delta_n - \gamma_n y(X_n-x^{\star})
   \Delta_n|^2 + \gamma_{n+1} |\delta M_{n+1}|^2 \\
   & \quad + \gamma_{n}^4 |(a_n I + b_n (A +y(X_n-x^{\star}) I)) \Delta_n|^2\\
   & \quad + 2 \gamma_n (\Delta_n - \gamma_n  Q \Delta_n   - \gamma_n
   y(X_n-x^{\star}) \Delta_n)' \delta M_{n+1} \\
   & \quad + 2 \gamma_n^4 ((a_n I + b_n (A +y(X_n-x^{\star}) I)) \Delta_n)'
   (\Delta_n - \gamma_n  Q \Delta_n   - \gamma_n
   y(X_n-x^{\star}) \Delta_n) \\
   & \quad + 2 \gamma_{n}^{5/2} \delta M_{n+1}' (a_n I + b_n (A +y(X_n-x^{\star})
   I)) \Delta M_n.
\end{align*}
If we take the conditional expectation with respect to
$\cf_n$ \textemdash\ denoted $\E_n$ \textemdash\ in the previous equality, we find
\begin{align*}
  \E_n\abs{\Delta_{n+\frac{1}{2}}}^2 \le 
  & | \Delta_n - \gamma_n Q \Delta_n - \gamma_n y(X_n-x^{\star})
   \Delta_n|^2 + \gamma_{n+1} \E_n |\delta M_{n+1}|^2 \\
   & \quad + \gamma_{n}^4  \norm{a_n I + b_n (A +y(X_n-x^{\star}) I)}^2 |\Delta_n|^2\\
   & \quad + 2 \gamma_n^4 \norm{(a_n I + b_n (A +y(X_n-x^{\star}) I))} (1 +
   \gamma_n  \norm{Q +  y(X_n-x^{\star})}) |\Delta_n|^2.
\end{align*}
\begin{align}
  \label{eq:half-norm}
  \E_n\left(\abs{\Delta_{n+\frac{1}{2}}}^2\right)  \leq & \abs{\Delta_n}^2 -2
  \gamma_{n} {\Delta_n}^\prime  (Q+y(X_n-x^{\star})) \Delta_n + \gamma_{n+1}
  \E_n|\delta M_{n+1}|^2 \nonumber\\
  & \quad + \cO\left(\gamma_n^2\right) (1+\abs{y(X_n-x^{\star})}) \abs{\Delta_n}^2.
\end{align}
Note that in the previous equation the quantity $\cO\left(\gamma_n^2\right)$ is non random.

Let $\lambda>0$ be the smallest eigenvalue of $Q$, which is symmetric definite
positive. Since $\lim_{\abs{x} \rightarrow 0} y(x) = 0$, there exists $\eta >
0$ such that for all $|x|<\eta$, $\abs{y(x)} < \lambda/2$. We assume that
this value of $\eta$ satisfies Hypothesis~\hypref{chen-mart-cv}.  Let
$\varepsilon>0$. Since $(X_n)_n$ converges almost surely to $x^\s$, there
exists a rank $N_0$ such that
\begin{equation*}
  \P(\sup_{m>N_0} \abs{X_m-x^{\star}} > \eta ) < \varepsilon.
\end{equation*}
Hence, $\P(A_n) \ge 1-\varepsilon$ for all $n>N_0$.

On the set $A_n$, $Q+y(X_n-x^{\star})$ is a positive definite matrix with
smallest eigenvalue greater than $\lambda/2$.  Therefore ${\Delta_n}^\prime (Q +
y(X_n-x^{\star}))\Delta_n > \lambda/2 \abs{\Delta_n}^2$. Hence, we can deduce
from Equation~\eqref{eq:half-norm} that
\begin{align*}
  \E\left(\abs{\Delta_{n+\frac{1}{2}}}^2 \ind{A_n}\right) -
  \E\left(\abs{\Delta_n}^2 \ind{A_n} \right) \leq &
  - \gamma_n \lambda \E\left(\abs{\Delta_n}^2 \ind{A_n}\right)
  + \gamma_n \kappa \\
  & + \cO(\gamma_n^{2}) (1 + \frac{1}{2} \lambda)
  \E\left(\abs{\Delta_n}^2 \ind{A_n}\right).
\end{align*}
We can assume that for $n>N_0, \quad |\cO(\gamma_n^2) (1 + \lambda/2)| \leq
\gamma_n \lambda/2$. Hence we get, for $n \ge N_0$,
\begin{align*}
  \E\left(\abs{\Delta_{n+\frac{1}{2}}}^2 \ind{A_n}\right) -
  \E\left(\abs{\Delta_n}^2 \ind{A_n} \right) \leq &
  - \gamma_n \frac{\lambda}{2} \E\left(\abs{\Delta_n}^2 \ind{A_n}\right)
  + \gamma_n \kappa
\end{align*}
Since $A_{n+1} \subset A_n$,
\begin{equation}
  \label{eq:83}
   \E\left(\abs{\Delta_{n+\frac{1}{2}}}^2 \ind{A_{n+1}}\right) -
  \E\left(\abs{\Delta_n}^2 \ind{A_n} \right)  \leq 
  - \gamma_n \frac{\lambda}{2} \E\left(\abs{\Delta_n}^2 \ind{A_n}\right) +  \kappa \gamma_n, 
\end{equation}
Now, we would like to replace $\Delta_{n+\frac{1}{2}}$ by $\Delta_{n+1}$ in
Equation~\eqref{eq:83}.
\begin{eqnarray*}
  \abs{\Delta_{n+1}}^2  & = & \frac{\abs{X_0 -
      x^\star}^2}{\gamma_{n+1}} \ind{p_{n+1} \neq 0} + 
  \abs{\Delta_{n+\frac{1}{2}}}^2 \ind{p_{n+1}=0},\\
  \abs{\Delta_{n+1}}^2  & \leq & \abs{\Delta_{n+\frac{1}{2}}}^2  +
  \frac{\abs{X_0 - x^\star}^2}{\gamma_{n+1}}   \ind{X_{n+\frac{1}{2}}
    \notin K_{\sigma_n}}.
\end{eqnarray*}
Taking the conditional expectation w.r.t. $\cf_n$ on the set $A_n$ gives
\begin{eqnarray}
  \label{eq:22}
  \E_n \abs{\Delta_{n+1}}^2  & \leq & \E_n \abs{\Delta_{n+\frac{1}{2}}}^2 +
  \frac{\abs{X_0 - x^\star}^2}{\gamma_{n+1}} 
  \P\left(X_{n+\frac{1}{2}} \notin 
    K_{\sigma_n} | \cf_n\right) , \nonumber \\
  \E_n \abs{\Delta_{n+1}}^2 \ind{A_n} & \leq &   \E_n \abs{\Delta_{n+\frac{1}{2}}}^2
  \ind{A_n} + \frac{\abs{X_0 - 
      x^\star}^2}{\gamma_{n+1}} \ind{A_n} \P\left(X_{n+\frac{1}{2}} \notin 
    K_{\sigma_n} | \cf_n\right),\nonumber \\
  \E \left(\abs{\Delta_{n+1}}^2 \ind{A_{n+1}}\right) & \leq & \E\left(
    \abs{\Delta_{n+\frac{1}{2}}}^2 \ind{A_n}\right) + \frac{\abs{X_0 -
        x^\star}^2}{\gamma_{n+1}} \P\left(A_n \cap \{X_{n+\frac{1}{2}}
        \notin K_{\sigma_n}\}  \right).
\end{eqnarray}
The probability on the right hand side can be rewritten
\begin{align*}
  \P\left(A_n \cap \{X_{n+\frac{1}{2}}
  \notin K_{\sigma_n}\}  \right) & = \E\left(  \ind{\gamma_{n+1} \abs{u(X_n)
      + \delta M_{n+1}} \geq d( X_n , \partial K_{\sigma_n}) }
  \ind{A_n} \right) 
\end{align*}
Moreover using the triangle inequality, we have $d( X_n ,\partial K_{\sigma_n})
\geq  d( x^{\star} ,\partial K_{\sigma_n}) - \abs{X_n -x^{\star}}$.  Due to
Hypothesis \hypref{chen-interior}, $d\left( x^{\star} ,\partial K_{\sigma_n}
\right) \ge \mu$ and on $A_n, \quad \abs{X_n -x^{\star}} \leq \eta$. Hence,
$d\left( X_n ,\partial K_{\sigma_n} \right)  \geq  \mu - \eta$.  One can choose
$\eta < \mu/2$ for instance, so that $d( X_n ,\partial K_{\sigma_n})
> \frac{\mu}{2}$. 
\begin{align}
  \label{eq:proba-proj}
  \P\left(A_n \cap \{X_{n+\frac{1}{2}} \notin K_{\sigma_n}\}  \right)  & \leq 
  \E \left(\E_n \left(\ind{\gamma_{n+1} \abs{u(X_n) + \delta M_{n+1}} \geq
    \frac{\mu}{2}}\right) \ind{A_n}\right) ,\nonumber \\
     & \leq \frac{8
    \gamma_{n+1}^2}{\mu^2} \E\left(\abs{u(X_n)}^2 \ind{A_n} + 
    \abs{\delta M_{n+1}}^2 ) \ind{A_n} \right).
  \end{align}
Thanks to Hypothesis~\hypref{chen-mart-cv} and the continuity of $u$, the
expectation on the r.h.s of~\eqref{eq:proba-proj} is bounded by a constant
$\bar c>0$ independent of $n$.  So, we get
\begin{equation*}
  \P\left(X_{n+\frac{1}{2}} \notin
      K_{\sigma_n}, A_n \right) \leq  \bar c \gamma_{n+1}^2.
\end{equation*}
Hence, from Equation (\ref{eq:22}) we can deduce
\begin{equation}
  \label{eq:11}
  \E \left(\abs{\Delta_{n+1}}^2 \ind{A_{n+1}}\right) \leq \E\left(
    \abs{\Delta_{n+\frac{1}{2}}}^2 \ind{A_n}\right) + \bar c \gamma_{n}.
\end{equation}
By combining Equations (\ref{eq:11}) and (\ref{eq:83}), we come up with
\begin{eqnarray*}
  \E \left(\abs{\Delta_{n+1}}^2 \ind{A_{n+1}}\right) & \leq &
  \left(1 - \gamma_n \frac{\lambda}{2}\right)
  \E\left(\abs{\Delta_n}^2 \ind{A_n}\right) + c \gamma_n,  
\end{eqnarray*}
where $c = \bar c + \kappa$.

Let $\ci = \left\{i \ge N_0:- \frac{\lambda}{2} \E\left(\abs{\Delta_i}^2
    \ind{A_i}\right) + c >0\right\}$, then
\begin{equation*}
  \sup_{i \in \ci} \E\left(\abs{\Delta_i}^2 \ind{A_i}\right) < \frac{2c}{\lambda} <
  \infty.
\end{equation*}
Note that we can always assume that $2c / \lambda \ge
\E\left(\abs{\Delta_{N_0}}^2 \ind{A_{N_0}}\right)$, such that the set $\ci$ is
non empty.
Assume $i \notin \ci$, let $i_0 = \sup\{ k < i \; : \; k \in \ci\}$. 
\begin{align*}
\E\left(\abs{\Delta_{i}}^2 \ind{A_{i}}\right) -
\E\left(\abs{\Delta_{i_0}}^2 \ind{A_{i_0}}\right) & \leq \sum_{k=i_0}^{i-1}
\gamma_k \left( c - \frac{\lambda}{2}  \E\left(\abs{\Delta_k}^2
\ind{A_k}\right)\right) 
\end{align*}
Since all the terms for $k=i_0+1, \dots, i-1$ are negative and $i_0 \in \ci$, we find
\begin{align*}
\E\left(\abs{\Delta_{i}}^2 \ind{A_{i}}\right) & \leq \gamma_{i_0} c +
\frac{2c}{\lambda}. 
\end{align*}
Finally, we come with the following upper bound.
\begin{equation*}
  \sup_{n \ge N_0} \E\left(\abs{\Delta_{n}}^2 \ind{A_n}\right) < \infty. 
\end{equation*}

\begin{rem}[case $1/2<\alpha<1$]
  \label{rem:alpha-less-1}
  This proof is still valid for $\alpha<1$ if we replace the Taylor expansions of
  Equation~(\ref{eq:8}) by
  \begin{equation*}
    \sqrt{\frac{\gamma_n}{\gamma_{n+1}}} =  1 + \frac{a_n}{n} \mbox{ and } 
    \sqrt{\gamma_n \gamma_{n+1}} = \gamma_n + \frac{\gamma_n b_n}{n}.
  \end{equation*}
  Then, Equation~(\ref{eq:3}) becomes
  \begin{align*}
    \qquad \Delta_{n+1} = & \Delta_n - \gamma_n  Q \Delta_n   - \gamma_n
    y(X_n-x^{\star}) \Delta_n  - \sqrt{\gamma_{n+1}}\delta M_{n+1} \\
    & \quad + \sqrt{\gamma_{n+1}}  p_{n+1} + \frac{1}{n}
    (a_n I + b_n \gamma_n (A + y(X_n-x^{\star})) \Delta_n, 
  \end{align*}
  with $Q=A$ this time, which is still positive definite.
\end{rem}

\subsubsection{Proof of Lemma \ref{lem:delta-n-expr}}

Let us go back to Equation (\ref{eq:3}). For any $n > N_0$ and $k > 0$, we can
write
\begin{align*}
   \Delta_{n+k} =  &  \Delta_{n+k-1} - \gamma_{n+k-1}  Q \Delta_{n+k-1}
   - \sqrt{\gamma_{n+k}}\delta M_{n+k} + \gamma_{n+k-1} \bar R_{n+k-1}
 \end{align*}
 where
 \begin{align*}
   \bar R_{m} = - y(X_{m}-x^{\star}) \Delta_{m}  
   & + \frac{1}{\sqrt{\gamma_{m+1}}}  p_{m+1} +  \gamma_{m}
  (a_{m} I + b_{m} (A +y(X_{m}-x^{\star}))) \Delta_{m}, 
\end{align*}
We can actually notice that the previous equation pretty much looks like a
discrete time ODE. Based on this remark, it is natural to multiply the previous
equation by $\expp{s_{n,k} Q}$ to find
\begin{align*}
   \expp{s_{n,k} Q}\Delta_{n+k} - ( \expp{s_{n,k} Q} - 
   \expp{s_{n,k} Q}\gamma_{n+k-1}  Q ) \Delta_{n+k-1} & = 
   - \expp{s_{n,k} Q} \sqrt{\gamma_{n+k}}\delta M_{n+k} + 
   \gamma_{n+k-1} \expp{s_{n,k} Q}\bar R_{n+k-1}
 \end{align*}
Note that $\expp{s_{n,k} Q} - \expp{s_{n,k} Q}\gamma_{n+k-1}  Q =
 \expp{s_{n,k-1}Q}( 1 + \cO (\gamma_{n+k-1}^2))$. Hence, we come up with the
following equation
\begin{align*}
   \expp{s_{n,k} Q}\Delta_{n+k} - \expp{s_{n,k-1} Q}  
   \Delta_{n+k-1} & = 
   - \expp{s_{n,k} Q} \sqrt{\gamma_{n+k}}\delta M_{n+k} + 
   \gamma_{n+k-1} \expp{s_{n,k} Q} R_{n+k-1}
 \end{align*}
where 
 \begin{align}
   \label{eq:reste}
   R_{m} = - y(X_{m}-x^{\star}) \Delta_{m}  
   & + \frac{1}{\sqrt{\gamma_{m+1}}}  p_{m+1} +  \gamma_{m}
  (a_{m} I + b_{m} (A +y(X_{m}-x^{\star})) + \cO(1)) \Delta_{m}, 
\end{align}
When summing the previous equalities for $k = 1,\dots, t-1$ for any integer
$t>0$, we get
\begin{equation*}
  \Delta_{n+t}  = \expp{-s_{n,t} Q} \Delta_n - \sum_{k=0}^{t-1}
  \expp{(s_{n,k}-s_{n,t}) Q} \sqrt{\gamma_{n+k+1}} \delta  M_{n+k+1}
  - \sum_{k=0}^{t-1} \expp{(s_{n,k}-s_{n,t}) Q} \gamma_{n+k} R_{n+k},
\end{equation*}
Let us a have a closer look at the different terms of Equation~\eqref{eq:reste}
\begin{itemize}
  \item $\lim_m y(X_m - x^\s) \Delta_m \ind{|X_m - x^\s| > \eta} = 0$ a.s.
    thanks to the a.s. convergence of $(X_m)_m$ and using Lemma~\ref{lem:tight},
    the sequence $(y(X_m - x^\s) \Delta_m \ind{|X_m - x^\s| \le \eta})_m$ is
    uniformly integrable and tends to zero in probability because $\lim_m y(X_m -
    x^\s) = 0$ a.s. 
  \item $p_m$ is almost surely equal to $0$ for $m$ large enough thanks to
    Remark~\ref{rem_hypo}, so $\frac{1}{\sqrt{\gamma_m}} p_m = 0$ a.s. for $m$
    large enough.
  \item $\gamma_{m} (a_{m} I + b_{m} (A +y(X_{m}-x^{\star})) + \cO(1))
    \Delta_{m} \ind{|X_m - x^\s| > \eta} \longrightarrow 0$ almost surely
    because for $m$ large enough the indicator equals $0$. The sequence
    $\gamma_{m} (a_{m} I + b_{m} (A +y(X_{m}-x^{\star})) + \cO(1)) \Delta_{m}
    \ind{|X_m - x^\s| \le \eta}$ is uniformly integrable by
    Lemma~\ref{lem:tight} and tends to zero in probability because $\gamma_m
    \longrightarrow 0$.
\end{itemize}
Hence, $R_m$ can be split in two terms : one tending to zero almost surely
and an other one which is uniformly integrable and tends to zero in probability.
Then, we can apply Propositions~\ref{prop:integral-convergence-ps}
and~\ref{prop:integral-convergence} to prove the convergence in probability of
$(\sum_{k=0}^{t-1} \expp{(s_{n,k}-s_{n,t}) Q} \gamma_{n+k} R_{n+k})_t$. This last
point ends the proof of Lemma~\ref{lem:delta-n-expr}.

\subsubsection{Proof of Lemma \ref{lem:int-sto-conv}}

To prove Lemma \ref{lem:int-sto-conv}, we need a result on the rate of
convergence of martingale arrays. First, note that for $\inv{\sqrt{\gamma_{n}}}
\delta  M_{n} \ind{\abs{X_{n-1} -x^\s}}$ tends to $0$ a.s. when $n$ goes to
infinity because $\ind{\abs{X_{n-1} -x^\s}} = 0$ for $n$ large enough. Then, it
ensues from Proposition~\ref{prop:integral-convergence-ps} that $\sum_{k=0}^t
\expp{Q(s_{n,k}-s_{n,t})} \sqrt{\gamma_{n+k}} \delta  M_{n+k}
\ind{\abs{X_{n+k-1} -x^\s} > \eta}$ converges to zero in probability when $t$
goes to infinity. Henceforth, it is sufficient to prove a localized version of
Lemma~\ref{lem:int-sto-conv} by considering $\sum_{k=0}^t
\expp{Q(s_{n,k}-s_{n,t})} \sqrt{\gamma_{n+k}} \delta  M_{n+k}
\ind{\abs{X_{n+k-1} -x^\s} \le \eta}$.  \\

We will use the following Central Limit Theorem for martingale arrays adapted
from \cite[Theorem 2.1.9]{duflo97:_random_iterat_model}.
\begin{thm}
  \label{thm:tcl-array} Suppose that $\{(\cf_l^{t})_{0 \leq l \leq t}; t >
  0\}$ is a family of filtrations and $\{(N_l^{t})_{0 \leq l \leq t}; t > 0\}$
  a square integrable martingale array with respect to the previous
  filtration. Assume that : 
  \begin{hypo}
    \label{hyp:hook}there exists a symmetric positive definite matrix $\Gamma$ such
    that $\langle N \rangle^{t}_t \xrightarrow[t \rightarrow \infty]{\P} \Gamma$.
  \end{hypo}
  \begin{hypo}
    \label{hyp:lindeberg}There exists $\rho>0$ such that
    \begin{equation*}
      \sum_{l=1}^t \E\left(\abs{N^{t}_l - N^{t}_{l-1}}^{2+\rho} \left|
          \cf^{t}_{l-1}\right.\right) \xrightarrow[t \rightarrow \infty]{\P} 0.
    \end{equation*}
  \end{hypo}
  Then,
  \begin{equation*}
    N^{t}_t \xrightarrow[t \rightarrow \infty]{\cl} \cn(0,\Gamma).
  \end{equation*}
\end{thm}
Using this theorem, we can now prove Lemma~\ref{lem:int-sto-conv}.
\begin{proof}[Proof of Lemma~\ref{lem:int-sto-conv}]
  For the sake of clearness, we will do the proof considering that $Q$ is a
  non-negative real constant instead of a positive definite matrix.  Let us
  define $N^t_l$ for all $0 \leq l \leq t$ and $t>0$
  \begin{equation*}
    N^t_l = \sum_{k=1}^{l}
    \expp{(s_{n,k}-s_{n,t}) Q} \sqrt{\gamma_{n+k}} \delta  M_{n+k}
    \ind{\abs{X_{n+k-1} -x^\s} \le \eta}.
  \end{equation*}
  $(N^t_l)_{0\leq l \leq p}$ is obviously a martingale  with respect to $(\cf_{n+l})_l$.
  Let us compute its angle bracket
  \begin{eqnarray}
    \label{eq:64}
    \langle N \rangle_t^t & = & \sum_{k=1}^{t} \expp{2 (s_{n,k}-s_{n,t}) Q}
    \gamma_{n+k} \E\left(\delta M_{n+k}^2 \ind{\abs{X_{n+k-1} -x^\s} \le \eta}|
      \cf_{n+k-1}\right).  
  \end{eqnarray}
  Thanks to Hypotheses \hypref{chen-mart-cv}, the conditional expectation in
  (\ref{eq:64}) is uniformly integrable and converges in probability to $\Sigma$
  when $k$ goes to infinity.  Applying
  Proposition~\ref{prop:integral-convergence}  proves the convergence in
  probability of $\langle N\rangle_t^t$ to $\lim_{t \rightarrow
  \infty}\sum_{k=1}^{t} \expp{2 (s_{n,k}-s_{n,t}) Q} \gamma_{n+k} \Sigma =
  \sum_{k=1}^{\infty} \expp{-2 s_{n,k} Q} \gamma_{n+k}  \Sigma$.  Let $\rho$ be
  the real number defined in Theorem~\ref{thm:tcl-chen}.
  \begin{equation}
    \label{eq:27}
    \sum_{l=1}^t \E\left(\abs{N^{t}_l - N^{t}_{l-1}}^{2+\rho} \right)  =
    \sum_{k=1}^{t} \expp{(2+\rho)(s_{n,k}-s_{n,t}) Q} 
    \gamma_{n+k}^{1+\frac{\rho}{2}} \E\left(\abs{\delta
        M_{n+k}}^{2+\rho} \ind{\abs{X_{n+k-1} -x^\s} \le \eta}\right).
  \end{equation}
  $ \gamma_{n+k}^{\frac{\rho}{2}}$ converges to $0$ when $k$ goes to
  infinity and the sequence of expectations is bounded using Hypothesis
  \hypref{chen-mart-cv}, so $\gamma_{n+k}^{\frac{\rho}{2}}
  \E\left(\abs{\delta M_{n+k}}^{2+\rho} \ind{\abs{X_{n+k-1} -x^\s} \le
      \eta}\right)$ tends to zero when $k$ goes to infinity.
  Proposition~\ref{prop:integral-convergence-ps} proves that the l.h.s. of
  Equation~\eqref{eq:27} tends to $0$ when $t$ goes to infinity.
  Hence, $ \sum_{l=1}^t \E\left(\abs{N^{t}_l - N^{p}_{l-1}}^{2+\rho} \left|
      \cf^{t}_{l-1}\right.\right)$ tends to zero in $\L^1$, and consequently
  in probability. Then, the Hypotheses of Theorem~\ref{thm:tcl-array} are
  satisfied.\\
  Finally, we have proved that
  \begin{equation*}
   \sum_{k=0}^t
   \expp{Q(s_{n,k}-s_{n,t})} \sqrt{\gamma_{n+k}} \delta  M_{n+k} \xrightarrow[t
   \rightarrow \infty]{law} \cn\left(0, \sum_{k=1}^{\infty} \expp{-2 s_{n,k} Q}
   \gamma_{n+k}  \Sigma\right).
  \end{equation*}
\end{proof}

\section{Conclusion}

In this work, we have proved  a Central Limit Theorem with rate
$\sqrt{\gamma_n}$ for randomly truncated stochastic algorithms under local
assumptions. We have also tried to clarify the proof of the convergence rate of
randomly truncated stochastic algorithms under assumptions which can be easily
verified in practice. The improvement brought by this new set of assumptions is
that all they should only be checked in a neighbourhood of the target value
$x^\s$, which means that in the case where $u(x) = \E(U(x,Z))$ the assumptions
can be reformulated in terms of some local regularity properties of $U$.

\appendix

\section{Some elementary results}

Here are two results used in the proofs of Theorems~\ref{thm:tcl-chen}
and~\ref{thm:tcl2-chen}. 

\begin{prop}
  \label{prop:integral-convergence-ps} Let $(Y_n)_n$ be a sequence of random
  vectors of $\R^d$ converging almost surely to a non random vector $x \in
  \R^d$. For any fixed integer $n>0$ and positive definite matrix $Q \in \R^{d
  \times d}$, we define, for all integers $t \geq 0$, $Z_t = \sum_{k=0}^t
  \expp{Q(s_{n,k}-s_{n,t})} \gamma_{n+k} Y_{n+k}$. Then, $\lim_t Z_t =
  \int_0^\infty \expp{-Q u} du \; x$ almost surely.
\end{prop}

\begin{proof} 
  It is clear that $\lim_{t \rightarrow \infty} \int_0^{s_{n,t}} \expp{-Q u} du
  \; x = \int_0^\infty \expp{-Q u} du \; x$. Hence, it is sufficient to consider  
  \begin{align}
    \abs{Z_t - \int_0^{s_{n,t}} \expp{-Q u} du \; x} \leq  & \sum_{k=0}^t
    \gamma_{n+k} \norm{\expp{Q(s_{n,k}-s_{n,t})}} |Y_{n+k} - x| \nonumber\\
    & \quad + \norm{\sum_{k=0}^t \gamma_{n+k} \expp{Q(s_{n,k}-s_{n,t})} -   
      \int_0^{s_{n,t}} \expp{-Q u} du} |x|. 
    \label{Zdecomp}
  \end{align}

  Let $\underbar q>0$ (resp. $\bar q>0$) be the smallest (resp. greatest)
  eigenvalue of $Q$.\\

  {\bf Step 1 : }~We will prove that the first term in Equation~\eqref{Zdecomp} tends
  to $0$ almost surely.
  \begin{align}
    \label{Z_term_1}
    \sum_{k=0}^t
    \gamma_{n+k} \norm{\expp{Q(s_{n,k}-s_{n,t})}} |Y_{n+k} - x| & \leq 
    \sum_{k=0}^t
    \int_{s_{n,k-1}}^{s_{n,k}} \expp{\underbar q(s_{n,k}-s_{n,t})} |Y_{n+k} - x| \;
    du \nonumber \\
    & \leq \int_0^{s_{n,t}} \expp{\underbar q(u-s_{n,t})} \expp{\underbar q
    \gamma_{n+\tau_n(u)}} |Y_{n+\tau_n(u)} - x| \; du
  \end{align}
  where for any real number $u>0$ $t_n(u)$ is the largest integer $k$ such that
  $s_{n,k-1} \le u < s_{n,k}$. Note that $\lim_{u \rightarrow +\infty} t_n(u) =
  +\infty$. $\lim_{u \rightarrow +\infty} \expp{\underbar q
  \gamma_{n+\tau_n(u)}} |Y_{n+\tau_n(u)} - x| = 0$ a.s., hence it is obvious
  that the term on the r.h.s of Equation~\eqref{Z_term_1} tend to $0$ almost surely.

  {\bf Step 2 : }~We will now prove that the second term in Equation~\eqref{Zdecomp}
  tends to $0$. \\
  We use the convention $s_{n,-1} = 0$ and recall that $s_{n,k} = s_{n,k-1} +
  \gamma_{n+k}$. Note that $\int_0^{s_{n,t}} \expp{-Q u} du =  \int_0^{s_{n,t}}
  \expp{Q (u - s_{n,t})} du$, hence the following inequality holds 
  \begin{align*}
    \norm{\sum_{k=0}^t \gamma_{n+k} \expp{Q(s_{n,k}-s_{n,t})} -   
      \int_0^{s_{n,t}} \expp{-Q u} du} & \leq  
    \sum_{k=0}^t \int_{s_{n,k-1}}^{s_{n,k}} \norm{\expp{Q(s_{n,k}-s_{n,t})}
      -\expp{Q (u-s_{n,t})}} du  \\
    & \leq \sum_{k=0}^t \int_{s_{n,k-1}}^{s_{n,k}} \norm{\expp{Q (u-s_{n,t})}}
    \norm{\expp{Q(s_{n,k}-u)} - I} du.  \\
    & \leq \sum_{k=0}^t \int_{s_{n,k-1}}^{s_{n,k}} \expp{\underbar q (u-s_{n,t})}
    (\expp{\bar q \gamma_{n+k}} - 1) du.
  \end{align*}
  Let $\varepsilon >0$, there exits $T_1 >0$ such that for all $t \ge T_1$, 
  $(\expp{q \gamma_{n+t}} - 1) \le \varepsilon$, hence for all $t > T_1$,
  \begin{align*}
    \sum_{k=0}^t \int_{s_{n,k-1}}^{s_{n,k}} \expp{\underbar q (u-s_{n,t})}
    (\expp{\bar q \gamma_{n+k}} - 1) du  & \leq
    \sum_{k=0}^{T_1} \int_{s_{n,k-1}}^{s_{n,k}} \expp{\underbar q (u-s_{n,t})}
    (\expp{\bar q} - 1) du +  
    \varepsilon \sum_{k=T_1+1}^t \int_{s_{n,k-1}}^{s_{n,k}} \expp{\underbar q
      (u-s_{n,t})} du, \\
    & \leq \int_{0}^{s_{n,T_1}} \expp{\underbar q (u-s_{n,t})} (\expp{\bar q} -
    1) du +  \varepsilon \int_{s_{n,T_1}}^{s_{n,t}} \expp{\underbar q
      (u-s_{n,t})} du, \\
    & \leq  (\expp{\underbar q ({s_{n,T_1}}-s_{n,t})} -  \expp{-\underbar q
      s_{n,t}}) \frac{\expp{\bar q} - 1}{\underbar q} +  \varepsilon
    \frac{1}{\underbar q}.\\
  \end{align*}
  There exists $T_2 > T_1$ such that for all $t>T_2$, $(\expp{\underbar q
    ({s_{n,T_1}}-s_{n,t})} -  \expp{-\underbar q s_{n,t}}) (\expp{\bar q} - 1) \le
  \varepsilon$, hence for all $t> T_2$,
  \begin{align*}
    \sum_{k=0}^t \int_{s_{n,k-1}}^{s_{n,k}} \expp{\underbar q (u-s_{n,t})}
    (\expp{\bar q \gamma_{n+k}} - 1) du \le \frac{2 \varepsilon}{\underbar q}.
  \end{align*}
  This ends to prove that the second term in Equation~\eqref{Zdecomp} tends to
  $0$ when $t$ goes to infinity.
  \end{proof}

\begin{prop}
  \label{prop:integral-convergence} 
  The conclusion of Proposition~\ref{prop:integral-convergence-ps} still holds
  if we assume that the sequence $(Y_n)_n$ is uniformly integrable and if it
  converges in probability to a non random vector $x \in \R^n$.
\end{prop}

\begin{proof} 
  We recall the decomposition given by Equation~\eqref{Zdecomp}
  \begin{align*}
    \abs{Z_t - \int_0^{s_{n,t}} \expp{-Q u} du \; x} \leq  & \sum_{k=0}^t
    \gamma_{n+k} \norm{\expp{Q(s_{n,k}-s_{n,t})}} |Y_{n+k} - x| \\
    & \quad + \norm{\sum_{k=0}^t \gamma_{n+k} \expp{Q(s_{n,k}-s_{n,t})} -   
      \int_0^{s_{n,t}} \expp{-Q u} du} |x|. 
  \end{align*}
  The last term in the above equation has already been proved to tend to $0$ in
  the proof of Proposition~\ref{prop:integral-convergence-ps} Step 2. So, we
  only need to prove that $\lim_{u \rightarrow +\infty}\sum_{k=0}^t
  \gamma_{n+k} \norm{\expp{Q(s_{n,k}-s_{n,t})}} |Y_{n+k} - x| = 0$ in
  probability. 

  Let $\underbar q>0$ (resp. $\bar q>0$) be the smallest (resp. greatest)
  eigenvalue of $Q$.\\

  \begin{align*}
    \sum_{k=0}^t
    \gamma_{n+k} \norm{\expp{Q(s_{n,k}-s_{n,t})}} |Y_{n+k} - x| & \leq 
    \sum_{k=0}^t
    \int_{s_{n,k-1}}^{s_{n,k}} \expp{\underbar q(s_{n,k}-s_{n,t})} |Y_{n+k} - x| \; du
    \\
    & \leq \int_0^{s_{n,t}} \expp{\underbar q(u-s_{n,t})} \expp{\underbar q
    \gamma_{n+\tau_n(u)}}  |Y_{n+\tau_n(u)} - x| \; du
  \end{align*}
  where for any real number $u>0$, $t_n(u)$ is the largest integer $k$ such that
  $s_{n,k-1} \le u < s_{n,k}$. Let $\bar Y_k = \gamma_k |Y_k - x|$. The sequence
  $(\bar Y_k)_k$ tends to zero in probability, is uniformly integrable and
  positive.
  \begin{align}
    \label{Zdecomp_2}
    \E\left(\int_0^{s_{n,t}} \expp{\underbar q(u-s_{n,t})} \bar Y_{n+\tau_n(u)}
    \; du \right) & = \int_0^{s_{n,t}} \expp{\underbar q(u-s_{n,t})} \E(\bar
    Y_{n+\tau_n(u)}) du.
  \end{align}
  Since $(\bar Y_k)_k$ is uniformly integrable and converges to $0$ in
  probability, $\lim_{u \rightarrow +\infty}\E(\bar Y_{n+\tau_n(u)}) = 0$, hence
  the term on the r.h.s of Equation~\eqref{Zdecomp_2} tends to $0$ when $t$ goes
  to infinity. This proves that $\lim_{u \rightarrow +\infty}\sum_{k=0}^t
  \gamma_{n+k} \norm{\expp{Q(s_{n,k}-s_{n,t})}} |Y_{n+k} - x| = 0$ in $\L^1$ and
  in probability. 
  \end{proof}

\end{document}